\documentclass{amsart}
\usepackage{amsfonts}

\newtheorem{theorem}{Theorem}
\theoremstyle{plain}

\newtheorem{definition}{Definition}

\newtheorem{lemma}{Lemma}

\newtheorem{remark}{Remark}

\numberwithin{equation}{section}
\input{tcilatex}

\begin{document}
\title[Hermite-Hadamard Inequality]{Hermite-Hadamard's inequalities for
prequasiinvex functions via fractional integrals}
\author{Imdat Iscan}
\address{Department of Mathematics, Faculty of Science and Arts, Giresun
University, Giresun, Turkey}
\email{imdat.iscan@giresun.edu.tr}
\date{March 24, 2012}

\begin{abstract}
In this paper, we extend some estimates of the right hand side of a
Hermite-Hadamard type inequality for prequasiinvex functions via fractional
integrals.
\end{abstract}

\maketitle

\section{Introduction and Preliminaries}

Let $f:I\subset \mathbb{R\rightarrow R}$ be a convex mapping defined on the
interval $I$ of real numbers and $a,b\in I$ with $a<b$, then

\begin{equation}
f\left( \frac{a+b}{2}\right) \leq \frac{1}{b-a}\dint\limits_{a}^{b}f(x)dx%
\leq \frac{f(a)+f(b)}{2}\text{.}  \label{1-1}
\end{equation}

This doubly inequality is known in the literature as Hermite-Hadamard
integral inequality for convex mapping. For several recent results
concerning the inequality (\ref{1-1}) we refer the interested reader to \cite%
{DP00,BOP08,SSO10,SSYB11}.

We recall that the notion of quasi-convex functions generalizes the notion
of convex functions. More precisely, a function $f:\left[ a,b\right] \mathbb{%
\rightarrow R}$ is said to be quasi-convex on $\left[ a,b\right] $ if
inequality%
\begin{equation*}
f\left( tx+(1-t)y\right) \leq \max \left\{ f(x),f(y)\right\} ,
\end{equation*}%
holds for all $x,y\in $ $I$ and $t\in \left[ 0,1\right] .$

Clearly, any convex function is quasi-convex function. Furthemore there
exist quasi-convex functions which are not convex (see \cite{I07}).

We give some necessary definitions and mathematical preliminaries of
fractional calculus theory which are used throughout this paper.

\begin{definition}
Let $f\in L\left[ a,b\right] $. The Riemann-Liouville integrals $%
J_{a^{+}}^{\alpha }f$ and $J_{b^{-}}^{\alpha }f$ of oder $\alpha >0$ with $%
a\geq 0$ are defined by

\begin{equation*}
J_{a^{+}}^{\alpha }f(x)=\frac{1}{\Gamma (\alpha )}\dint\limits_{a}^{x}\left(
x-t\right) ^{\alpha -1}f(t)dt,\ x>a
\end{equation*}

and

\begin{equation*}
J_{b^{-}}^{\alpha }f(x)=\frac{1}{\Gamma (\alpha )}\dint\limits_{x}^{b}\left(
t-x\right) ^{\alpha -1}f(t)dt,\ x<b
\end{equation*}

respectively, where $\Gamma (\alpha )$ is the Gamma function and $%
J_{a^{+}}^{0}f(x)=J_{b^{-}}^{0}f(x)=f(x).$
\end{definition}

In the case of $\alpha =1$, the fractional integral reduces to the classical
integral. For some recent result connected with fractional integral
inequalities see (\cite{SSYB11}-\cite{OY11}).

In \cite{OY11}, Ozdemir and Y\i ld\i z proved the Hadamard inequality for
quasi-convex functions via Riemann-Liouville fractional integrals as follows:

\begin{theorem}
\label{1.a}Let $f:\left[ a,b\right] \mathbb{\rightarrow R}$, be positive
function with $0\leq a<b$ and $f\in L\left[ a,b\right] .$ If $f$ is a
quasi-convex function on $\left[ a,b\right] $, then the following inequality
for fractional integrals holds:%
\begin{equation}
\frac{\Gamma (\alpha +1)}{2(b-a)^{\alpha }}\left[ J_{a^{+}}^{\alpha
}f(b)+J_{b^{-}}^{\alpha }f(a)\right] \leq \max \left\{ f(a),f(b)\right\}
\label{1-2}
\end{equation}%
with $\alpha >0.$
\end{theorem}

\begin{theorem}
\label{1.c}Let $f:\left[ a,b\right] \mathbb{\rightarrow R}$, be a
differentiable mapping on $\left( a,b\right) $ with $a<b.$ If $\left\vert
f^{\prime }\right\vert $ is quasi-convex on $\left[ a,b\right] ,\ \alpha >0,$
then the following inequality for fractional integrals holds:%
\begin{eqnarray}
&&\left\vert \frac{f(a)+f(b)}{2}-\frac{\Gamma (\alpha +1)}{2(b-a)^{\alpha }}%
\left[ J_{a^{+}}^{\alpha }f(b)+J_{b^{-}}^{\alpha }f(a)\right] \right\vert
\label{1-3} \\
&\leq &\frac{b-a}{\alpha +1}\left( 1-\frac{1}{2^{\alpha }}\right) \max
\left\{ \left\vert f^{\prime }(a)\right\vert ,\left\vert f^{\prime
}(b)\right\vert \right\} .  \notag
\end{eqnarray}
\end{theorem}

\begin{theorem}
\label{1.d}Let $f:\left[ a,b\right] \mathbb{\rightarrow R}$, be a
differentiable mapping on $\left( a,b\right) $ with $a<b$ such that $%
f^{\prime }\in $ $L\left[ a,b\right] .$ If $\left\vert f^{\prime
}\right\vert ^{q}$ is quasi-convex on $\left[ a,b\right] ,$ and $q>1,$ then
the following inequality for fractional integrals holds:%
\begin{eqnarray}
&&\left\vert \frac{f(a)+f(b)}{2}-\frac{\Gamma (\alpha +1)}{2(b-a)^{\alpha }}%
\left[ J_{a^{+}}^{\alpha }f(b)+J_{b^{-}}^{\alpha }f(a)\right] \right\vert
\label{1-4} \\
&\leq &\frac{b-a}{2\left( \alpha p+1\right) ^{\frac{1}{p}}}\left( \max
\left\{ \left\vert f^{\prime }(a)\right\vert ^{q},\left\vert f^{\prime
}(b)\right\vert ^{q}\right\} \right) ^{\frac{1}{q}}  \notag
\end{eqnarray}%
where $\frac{1}{p}+\frac{1}{q}=1$ and $\alpha \in \left[ 0,1\right] .$
\end{theorem}

\begin{theorem}
\label{1.e}Let $f:\left[ a,b\right] \mathbb{\rightarrow R}$, be a
differentiable mapping on $\left( a,b\right) $ with $a<b$ such that $%
f^{\prime }\in $ $L\left[ a,b\right] .$ If $\left\vert f^{\prime
}\right\vert ^{q}$ is quasi-convex on $\left[ a,b\right] ,$ and $q\geq 1,$
then the following inequality for fractional integrals holds:%
\begin{eqnarray}
&&\left\vert \frac{f(a)+f(b)}{2}-\frac{\Gamma (\alpha +1)}{2(b-a)^{\alpha }}%
\left[ J_{a^{+}}^{\alpha }f(b)+J_{b^{-}}^{\alpha }f(a)\right] \right\vert
\label{1-44} \\
&\leq &\frac{b-a}{\left( \alpha +1\right) }\left( 1-\frac{1}{2^{\alpha }}%
\right) \left( \max \left\{ \left\vert f^{\prime }(a)\right\vert
^{q},\left\vert f^{\prime }(b)\right\vert ^{q}\right\} \right) ^{\frac{1}{q}}
\notag
\end{eqnarray}%
with $\alpha >0.$
\end{theorem}

\bigskip

In recent years several extentions and generalizations have been considered
for classical convexity. A significant generalization of convex functions is
that of invex functions introduced by Hanson in \cite{H81}. Weir and Mond 
\cite{WM98} introduced the concept of preinvex functions and applied it to
the establisment of the sufficient optimality conditions and duality in
nonlinear programming. Pini \cite{P91} introduced the concept of
prequasiinvex function as a generalization of invex functions. Later, Mohan
and Neogy \cite{MN95} obtained some properties of generalized preinvex
functions. Noor \cite{N09}-\cite{N07} has established some Hermite-Hadamard
type inequalities for preinvex and log-preinvex functions.In recent papers
Yang et al.in \cite{YYL01} studied prequasiinvex functions, and semistrictly
prequasiinvex functions and Barani et al. in \cite{BGD11} presented some
generaliztions of the right hand side of a Hermite-Hadamard type inequality
for prequasiinvex functions.

In this paper we generalized the results in \cite{OY11} for prequasiinvex
functions. Now we recall some notions in invexity analysis which will be
used throught the paper (see \cite{A05,YL01} and references therein)

Let $f:A\mathbb{\rightarrow R}$ and $\eta :A\times A\rightarrow 
\mathbb{R}
,$where $A$ is a nonempty set in $%
\mathbb{R}
^{n}$, be continuous functions.

\begin{definition}
The set $A\subseteq $ $%
\mathbb{R}
^{n}$ is said to be invex with respect to $\eta (.,.)$, if for every $x,y\in
A$ and $t\in \left[ 0,1\right] ,$%
\begin{equation*}
x+t\eta (y,x)\in A.
\end{equation*}

The invex set $A$ is also called a $\eta -$connected set.
\end{definition}

It is obvious that every convex set is invex with respect to $\eta (y,x)=y-x$%
, but there exist invex sets which are not convex \cite{A05}.

\begin{definition}
The function $f$ on the invex set $A$ is said to be preinvex with respect to 
$\eta $ if 
\begin{equation*}
f\left( x+t\eta (y,x)\right) \leq \left( 1-t\right) f(x)+tf(y),\ \forall
x,y\in A,\ t\in \left[ 0,1\right] .
\end{equation*}

The function $f$ is said to be preconcave if and only if $-f$ \ is preinvex.
\end{definition}

\begin{definition}
The function $f$ on the invex set $A$ is said to be prequasiinvex with
respect to $\eta $ if 
\begin{equation*}
f\left( x+t\eta (y,x)\right) \leq \max \left\{ f(x),f(y)\right\} ,\ \forall
x,y\in A,\ t\in \left[ 0,1\right] .
\end{equation*}
\end{definition}

Every quasi-convex function is a prequasinvex with respect to $\eta
(y,x)=y-x,$ but the converse does not holds (see example 1.1 in \cite{YYL01})

We also need the following assumption regarding the function $\eta $ which
is due to Mohan and Neogy \cite{MN95}:

\textbf{Condition C:} Let $A\subseteq $ $%
\mathbb{R}
^{n}$ be an open invex subset with respect to $\eta :A\times A\rightarrow 
\mathbb{R}
.$ For any $x,y\in A$ and any $t\in \left[ 0,1\right] ,$%
\begin{eqnarray*}
\eta \left( y,y+t\eta (x,y)\right) &=&-t\eta (x,y) \\
\eta \left( x,y+t\eta (x,y)\right) &=&(1-t)\eta (x,y).
\end{eqnarray*}

Note that for every $x,y\in A$ and every $t\in \left[ 0,1\right] $ from
condition C, we have

\begin{equation}
\eta \left( y+t_{2}\eta (x,y),y+t_{1}\eta (x,y)\right) =(t_{2}-t_{1})\eta
(x,y).  \label{1-5}
\end{equation}

In \cite{BGD11} Barani et al. proved the Hermite-Hadamard type inequality
for prequasiinvex as follows:

\begin{theorem}
\label{1.3}Let $A\subseteq $ $%
\mathbb{R}
$ be an open invex subset with respect to $\eta :A\times A\rightarrow 
\mathbb{R}
.$ Suppose that $f:A\rightarrow 
\mathbb{R}
$ is a differentiable function. If $\left\vert f^{\prime }\right\vert $ is
prequasiinvex on $A$ then, for every $a,b\in A$ the following inequalities
holds%
\begin{eqnarray}
&&\left\vert \frac{f(a)+f\left( a+\eta (b,a)\right) }{2}-\frac{1}{\eta (b,a)}%
\dint\limits_{a}^{a+\eta (b,a)}f(x)dx\right\vert  \notag \\
&\leq &\frac{\left\vert \eta (b,a)\right\vert }{4}\max \left\{ \left\vert
f^{\prime }(a)\right\vert ,\left\vert f^{\prime }(b)\right\vert \right\} .
\label{1-6}
\end{eqnarray}
\end{theorem}

\begin{theorem}
\label{1.b}Let $A\subseteq $ $%
\mathbb{R}
$ be an open invex subset with respect to $\eta :A\times A\rightarrow 
\mathbb{R}
.$ Suppose that $f:A\rightarrow 
\mathbb{R}
$ is a differentiable function. Assume that $p\in 
\mathbb{R}
$ with $p>1.$ If $\left\vert f^{\prime }\right\vert ^{\frac{p}{p-1}}$ is
preinvex on $A$ then, for every $a,b\in A$ the following inequalities holds 
\begin{eqnarray}
&&\left\vert \frac{f(a)+f\left( a+\eta (b,a)\right) }{2}-\frac{1}{\eta (b,a)}%
\dint\limits_{a}^{a+\eta (b,a)}f(x)dx\right\vert  \notag \\
&\leq &\frac{\left\vert \eta (b,a)\right\vert }{2(p+1)^{\frac{1}{p}}}\left(
\max \left\{ \left\vert f^{\prime }(a)\right\vert ^{\frac{p}{p-1}%
},\left\vert f^{\prime }(b)\right\vert ^{\frac{p}{p-1}}\right\} \right) ^{%
\frac{p-1}{p}}.  \label{1-7}
\end{eqnarray}
\end{theorem}

In \cite{I12}, Iscan proved the following Lemma and established some
inequalities for preinvexfunctions via fractional integrals

\begin{lemma}
\label{1.4}Let $A\subseteq $ $%
\mathbb{R}
$ be an open invex subset with respect to $\eta :A\times A\rightarrow 
\mathbb{R}
$ and $a,b\in A$ with $a<a+\eta (b,a).$ Suppose that $f:A\rightarrow 
\mathbb{R}
$ is a differentiable function. If $f^{\prime }$ is preinvex function on $A$
and $f^{\prime }\in L\left[ a,a+\eta (b,a)\right] $ then, the following
equality holds:%
\begin{eqnarray}
&&\frac{f(a)+f\left( a+\eta (b,a)\right) }{2}-\frac{\Gamma (\alpha +1)}{%
2\eta ^{\alpha }(b,a)}\left[ J_{a^{+}}^{\alpha }f(a+\eta (b,a))+J_{\left(
a+\eta (b,a)\right) ^{-}}^{\alpha }f(a)\right]  \label{1-8} \\
&=&\frac{\eta (b,a)}{2}\dint\limits_{0}^{1}\left[ t^{\alpha }-\left(
1-t\right) ^{\alpha }\right] f^{\prime }\left( a+t\eta (b,a)\right) dt 
\notag
\end{eqnarray}
\end{lemma}

In this paper, using lemma \ref{1.4} we obtained new inequalities related to
the right side of Hermite-Hadamard inequalities for prequasiinvex functions
via fractional integrals.

\section{Main Results}

\begin{theorem}
\label{2.1}Let $A\subseteq $ $%
\mathbb{R}
$ be an open invex subset with respect to $\eta :A\times A\rightarrow 
\mathbb{R}
$ and $a,b\in A$ with $a<a+\eta (b,a).$ If $f:\left[ a,a+\eta (b,a)\right]
\rightarrow \left( 0,\infty \right) $ is a prequasiinvex function, $f\in L%
\left[ a,a+\eta (b,a)\right] $ and $\eta $ satisfies condition C then, the
following inequalities for fractional integrals holds:%
\begin{eqnarray}
&&\frac{\Gamma (\alpha +1)}{2\eta ^{\alpha }(b,a)}\left[ J_{a^{+}}^{\alpha
}f(a+\eta (b,a))+J_{\left( a+\eta (b,a)\right) ^{-}}^{\alpha }f(a)\right] 
\notag \\
&\leq &\max \left\{ f(a),f(a+\eta (b,a)\right\} \leq \max \left\{
f(a),f(b)\right\}  \label{2}
\end{eqnarray}

\begin{proof}
Since $a,b\in A$ and $A$ is an invex set with respect to $\eta $, for every $%
t\in \left[ 0,1\right] $, we have $a+t\eta (b,a)\in A.$ By prequasiinvexity
of $f$ and inequality (\ref{1-5}) for every $t\in \left[ 0,1\right] $ \ we
get 
\begin{eqnarray}
f\left( a+t\eta (b,a)\right) &=&f\left( a+\eta (b,a)+(1-t)\eta (a,a+\eta
(b,a))\right)  \notag \\
&\leq &\max \left\{ f(a),f(a+\eta (b,a)\right\}  \label{2-11}
\end{eqnarray}%
and similarly%
\begin{eqnarray*}
f\left( a+(1-t)\eta (b,a)\right) &=&f\left( a+\eta (b,a)+t\eta (a,a+\eta
(b,a))\right) \\
&\leq &\max \left\{ f(a),f(a+\eta (b,a)\right\} .
\end{eqnarray*}%
By adding these inequalities we have%
\begin{equation}
f\left( a+t\eta (b,a)\right) +f\left( a+(1-t)\eta (b,a)\right) \leq 2\max
\left\{ f(a),f(a+\eta (b,a)\right\}  \label{2-2}
\end{equation}%
Then multiplying both (\ref{2-2}) by $t^{\alpha -1}$ and integrating the
resulting inequality with respect to $t$ over $\left[ 0,1\right] ,$ we obtain%
\begin{equation*}
\dint\limits_{0}^{1}t^{\alpha -1}f\left( a+t\eta (b,a)\right)
dt+\dint\limits_{0}^{1}t^{\alpha -1}f\left( a+(1-t)\eta (b,a)\right) dt\leq
2\max \left\{ f(a),f(a+\eta (b,a)\right\} \dint\limits_{0}^{1}t^{\alpha
-1}dt.
\end{equation*}%
i.e.%
\begin{equation*}
\frac{\Gamma (\alpha )}{\eta ^{\alpha }(b,a)}\left[ J_{a^{+}}^{\alpha
}f(a+\eta (b,a))+J_{\left( a+\eta (b,a)\right) ^{-}}^{\alpha }f(a)\right]
\leq \frac{2\max \left\{ f(a),f(a+\eta (b,a)\right\} }{\alpha }.
\end{equation*}%
Using the mapping $\eta $ satisfies condition C the proof is completed.
\end{proof}
\end{theorem}

\begin{remark}
If in Theorem \ref{2.1}, we let $\eta (b,a)=b-a$, then inequality (\ref{2})
become inequality (\ref{1-2}) of Theorem \ref{1.a}.

\begin{theorem}
\label{2.2}Let $K\subseteq $ $%
\mathbb{R}
$ be an open invex subset with respect to $\eta :K\times K\rightarrow 
\mathbb{R}
$ and $a,b\in K$ with $a<a+\eta (b,a)$ such that $f^{\prime }\in L\left[
a,a+\eta (b,a)\right] $. Suppose that $f:K\rightarrow 
\mathbb{R}
$ is a differentiable function. If $\left\vert f^{\prime }\right\vert $ is
prequasiinvex function on $A$ then the following inequality for fractional
integrals with $\alpha >0$ holds:%
\begin{eqnarray}
&&\left\vert \frac{f(a)+f\left( a+\eta (b,a)\right) }{2}-\frac{\Gamma
(\alpha +1)}{2\eta ^{\alpha }(b,a)}\left[ J_{a^{+}}^{\alpha }f(a+\eta
(b,a))+J_{\left( a+\eta (b,a)\right) ^{-}}^{\alpha }f(a)\right] \right\vert
\label{2-3} \\
&\leq &\frac{\eta (b,a)}{\alpha +1}\left( 1-\frac{1}{2^{\alpha }}\right)
\max \left\{ \left\vert f^{\prime }(a)\right\vert ,\left\vert f^{\prime
}(b)\right\vert \right\} .  \notag
\end{eqnarray}

\begin{proof}
Using lemma \ref{1.4} and the prequasiinvexity of $\left\vert f^{\prime
}\right\vert $ we get%
\begin{eqnarray*}
&&\left\vert \frac{f(a)+f\left( a+\eta (b,a)\right) }{2}-\frac{\Gamma
(\alpha +1)}{2\eta ^{\alpha }(b,a)}\left[ J_{a^{+}}^{\alpha }f(a+\eta
(b,a))+J_{\left( a+\eta (b,a)\right) ^{-}}^{\alpha }f(a)\right] \right\vert
\\
&\leq &\frac{\eta (b,a)}{2}\dint\limits_{0}^{1}\left\vert t^{\alpha }-\left(
1-t\right) ^{\alpha }\right\vert \left\vert f^{\prime }\left( a+t\eta
(b,a)\right) \right\vert dt \\
&\leq &\frac{\eta (b,a)}{2}\dint\limits_{0}^{1}\left\vert t^{\alpha }-\left(
1-t\right) ^{\alpha }\right\vert \max \left\{ \left\vert f^{\prime
}(a)\right\vert ,\left\vert f^{\prime }(b)\right\vert \right\} dt \\
&\leq &\frac{\eta (b,a)}{2}\left\{ \dint\limits_{0}^{\frac{1}{2}}\left[
\left( 1-t\right) ^{\alpha }-t^{\alpha }\right] \max \left\{ \left\vert
f^{\prime }(a)\right\vert ,\left\vert f^{\prime }(b)\right\vert \right\}
dt+\dint\limits_{\frac{1}{2}}^{1}\left[ t^{\alpha }-\left( 1-t\right)
^{\alpha }\right] \max \left\{ \left\vert f^{\prime }(a)\right\vert
,\left\vert f^{\prime }(b)\right\vert \right\} dt\right\} \\
&=&\eta (b,a)\max \left\{ \left\vert f^{\prime }(a)\right\vert ,\left\vert
f^{\prime }(b)\right\vert \right\} \left( \dint\limits_{0}^{\frac{1}{2}}%
\left[ \left( 1-t\right) ^{\alpha }-t^{\alpha }\right] dt\right) \\
&=&\frac{\eta (b,a)}{\left( \alpha +1\right) }\left( 1-\frac{1}{2^{\alpha }}%
\right) \max \left\{ \left\vert f^{\prime }(a)\right\vert ,\left\vert
f^{\prime }(b)\right\vert \right\} ,
\end{eqnarray*}%
which completes the proof.
\end{proof}
\end{theorem}
\end{remark}

\begin{remark}
a) If in Theorem \ref{2.2}, we let $\eta (b,a)=b-a$, then inequality (\ref%
{2-3}) become inequality (\ref{1-3}) of Theorem \ref{1.c}

b) If in Theorem\ref{2.2}, we let $\alpha =1$, then inequality (\ref{2-3})
become inequality (\ref{1-6}) of Theorem \ref{1.3}.

c) In Theorem\ref{2.2}, assume that $\eta $ satisfies condition C.Using
inequality (\ref{2-11}) we get%
\begin{eqnarray*}
&&\left\vert \frac{f(a)+f\left( a+\eta (b,a)\right) }{2}-\frac{\Gamma
(\alpha +1)}{2\eta ^{\alpha }(b,a)}\left[ J_{a^{+}}^{\alpha }f(a+\eta
(b,a))+J_{\left( a+\eta (b,a)\right) ^{-}}^{\alpha }f(a)\right] \right\vert
\\
&\leq &\frac{\eta (b,a)}{\left( \alpha +1\right) }\left( 1-\frac{1}{%
2^{\alpha }}\right) \max \left\{ \left\vert f^{\prime }(a)\right\vert
,\left\vert f^{\prime }(a+\eta (b,a))\right\vert \right\}
\end{eqnarray*}
\end{remark}

\begin{theorem}
\label{2.5}Let $A\subseteq $ $%
\mathbb{R}
$ be an open invex subset with respect to $\eta :A\times A\rightarrow 
\mathbb{R}
$ and $a,b\in A$ with $a<a+\eta (b,a)$ such that $f^{\prime }\in L\left[
a,a+\eta (b,a)\right] $. Suppose that $f:A\rightarrow 
\mathbb{R}
$ is a differentiable function. If $\left\vert f^{\prime }\right\vert ^{q}$
is prequasiinvex function on $A$ for some fixed $q>1$ then the following
inequality holds:%
\begin{eqnarray}
&&\left\vert \frac{f(a)+f\left( a+\eta (b,a)\right) }{2}-\frac{\Gamma
(\alpha +1)}{2\eta ^{\alpha }(b,a)}\left[ J_{a^{+}}^{\alpha }f(a+\eta
(b,a))+J_{\left( a+\eta (b,a)\right) ^{-}}^{\alpha }f(a)\right] \right\vert
\label{2-9} \\
&\leq &\frac{\eta (b,a)}{\left( \alpha +1\right) }\left( 1-\frac{1}{%
2^{\alpha }}\right) \left( \max \left\{ \left\vert f^{\prime }(a)\right\vert
^{q},\left\vert f^{\prime }(b)\right\vert ^{q}\right\} \right) ^{\frac{1}{q}}
\notag
\end{eqnarray}%
where $\frac{1}{p}+\frac{1}{q}=1$ and $\alpha >0.$
\end{theorem}

\begin{proof}
From lemma\ref{1.4} and using H\"{o}lder inequality with properties of
modulus, we have%
\begin{eqnarray*}
&&\left\vert \frac{f(a)+f\left( a+\eta (b,a)\right) }{2}-\frac{\Gamma
(\alpha +1)}{2\eta ^{\alpha }(b,a)}\left[ J_{a^{+}}^{\alpha }f(a+\eta
(b,a))+J_{\left( a+\eta (b,a)\right) ^{-}}^{\alpha }f(a)\right] \right\vert
\\
&\leq &\frac{\eta (b,a)}{2}\dint\limits_{0}^{1}\left\vert t^{\alpha }-\left(
1-t\right) ^{\alpha }\right\vert ^{\frac{1}{p}+\frac{1}{q}}\left\vert
f^{\prime }\left( a+t\eta (b,a)\right) \right\vert dt \\
&\leq &\frac{\eta (b,a)}{2}\left( \dint\limits_{0}^{1}\left\vert t^{\alpha
}-\left( 1-t\right) ^{\alpha }\right\vert dt\right) ^{\frac{1}{p}}\left(
\dint\limits_{0}^{1}\left\vert t^{\alpha }-\left( 1-t\right) ^{\alpha
}\right\vert \left\vert f^{\prime }\left( a+t\eta (b,a)\right) \right\vert
^{q}dt\right) ^{\frac{1}{q}}.
\end{eqnarray*}%
On the other hand, we have%
\begin{eqnarray*}
\dint\limits_{0}^{1}\left\vert t^{\alpha }-\left( 1-t\right) ^{\alpha
}\right\vert dt &=&\dint\limits_{0}^{\frac{1}{2}}\left[ \left( 1-t\right)
^{\alpha }-t^{\alpha }\right] dt+\dint\limits_{\frac{1}{2}}^{1}\left[
t^{\alpha }-\left( 1-t\right) ^{\alpha }\right] dt \\
&=&\frac{2}{\alpha +1}\left( 1-\frac{1}{2^{\alpha }}\right) .
\end{eqnarray*}%
Since $\left\vert f^{\prime }\right\vert ^{q}$ is prequasiinvex function on $%
A$, we obtain%
\begin{equation*}
\left\vert f^{\prime }\left( a+t\eta (b,a)\right) \right\vert ^{q}\leq \max
\left\{ \left\vert f^{\prime }(a)\right\vert ^{q},\left\vert f^{\prime
}(b)\right\vert ^{q}\right\} ,\ \ t\in \left[ 0,1\right]
\end{equation*}%
and 
\begin{eqnarray*}
\dint\limits_{0}^{1}\left\vert t^{\alpha }-\left( 1-t\right) ^{\alpha
}\right\vert \left\vert f^{\prime }\left( a+t\eta (b,a)\right) \right\vert
^{q}dt &\leq &\dint\limits_{0}^{1}\left\vert t^{\alpha }-\left( 1-t\right)
^{\alpha }\right\vert \max \left\{ \left\vert f^{\prime }(a)\right\vert
^{q},\left\vert f^{\prime }(b)\right\vert ^{q}\right\} dt \\
&=&\max \left\{ \left\vert f^{\prime }(a)\right\vert ^{q},\left\vert
f^{\prime }(b)\right\vert ^{q}\right\} \left( \dint\limits_{0}^{\frac{1}{2}}%
\left[ \left( 1-t\right) ^{\alpha }-t^{\alpha }\right] dt+\dint\limits_{%
\frac{1}{2}}^{1}\left[ t^{\alpha }-\left( 1-t\right) ^{\alpha }\right]
\right) \\
&=&\frac{2}{\alpha +1}\left( 1-\frac{1}{2^{\alpha }}\right) \max \left\{
\left\vert f^{\prime }(a)\right\vert ^{q},\left\vert f^{\prime
}(b)\right\vert ^{q}\right\}
\end{eqnarray*}%
from here we obtain inequality (\ref{2-9}) which completes the proof.
\end{proof}

\begin{remark}
a) If in Theorem\ref{2.5}, we let $\eta (b,a)=b-a$ then inequality (\ref{2-9}%
)become inequality (\ref{1-44}) Theorem\ref{1.e}.

b) In Theorem\ref{2.5}, assume that $\eta $ satisfies condition C.Using
inequality (\ref{2-11}) we get%
\begin{eqnarray*}
&&\left\vert \frac{f(a)+f\left( a+\eta (b,a)\right) }{2}-\frac{\Gamma
(\alpha +1)}{2\eta ^{\alpha }(b,a)}\left[ J_{a^{+}}^{\alpha }f(a+\eta
(b,a))+J_{\left( a+\eta (b,a)\right) ^{-}}^{\alpha }f(a)\right] \right\vert
\\
&\leq &\frac{\eta (b,a)}{\left( \alpha +1\right) }\left( 1-\frac{1}{%
2^{\alpha }}\right) \left( \max \left\{ \left\vert f^{\prime }(a)\right\vert
^{q},\left\vert f^{\prime }(a+\eta (b,a))\right\vert ^{q}\right\} \right) ^{%
\frac{1}{q}}.
\end{eqnarray*}

\begin{theorem}
\label{2.4}Let $A\subseteq $ $%
\mathbb{R}
$ be an open invex subset with respect to $\eta :A\times A\rightarrow 
\mathbb{R}
$ and $a,b\in A$ with $a<a+\eta (b,a)$ such that $f^{\prime }\in L\left[
a,a+\eta (b,a)\right] $. Suppose that $f:A\rightarrow 
\mathbb{R}
$ is a differentiable function. If $\left\vert f^{\prime }\right\vert ^{q}$
is prequasiinvex function on $A$ for some fixed $q>1$ then the following
inequality holds:%
\begin{eqnarray}
&&\left\vert \frac{f(a)+f\left( a+\eta (b,a)\right) }{2}-\frac{\Gamma
(\alpha +1)}{2\eta ^{\alpha }(b,a)}\left[ J_{a^{+}}^{\alpha }f(a+\eta
(b,a))+J_{\left( a+\eta (b,a)\right) ^{-}}^{\alpha }f(a)\right] \right\vert
\label{2-8} \\
&\leq &\frac{\eta (b,a)}{2\left( \alpha p+1\right) ^{\frac{1}{p}}}\left(
\max \left\{ \left\vert f^{\prime }(a)\right\vert ^{q},\left\vert f^{\prime
}(a+\eta (b,a))\right\vert ^{q}\right\} \right) ^{\frac{1}{q}}.  \notag
\end{eqnarray}%
where $\frac{1}{p}+\frac{1}{q}=1$ and $\alpha \in \left[ 0,1\right] .$
\end{theorem}
\end{remark}

\begin{proof}
From lemma\ref{1.4} and using H\"{o}lder inequality with properties of
modulus, we have%
\begin{eqnarray*}
&&\left\vert \frac{f(a)+f\left( a+\eta (b,a)\right) }{2}-\frac{\Gamma
(\alpha +1)}{2\eta ^{\alpha }(b,a)}\left[ J_{a^{+}}^{\alpha }f(a+\eta
(b,a))+J_{\left( a+\eta (b,a)\right) ^{-}}^{\alpha }f(a)\right] \right\vert
\\
&\leq &\frac{\eta (b,a)}{2}\dint\limits_{0}^{1}\left\vert t^{\alpha }-\left(
1-t\right) ^{\alpha }\right\vert \left\vert f^{\prime }\left( a+t\eta
(b,a)\right) \right\vert dt \\
&\leq &\frac{\eta (b,a)}{2}\left( \dint\limits_{0}^{1}\left\vert t^{\alpha
}-\left( 1-t\right) ^{\alpha }\right\vert ^{p}dt\right) ^{\frac{1}{p}}\left(
\dint\limits_{0}^{1}\left\vert f^{\prime }\left( a+t\eta (b,a)\right)
\right\vert ^{q}dt\right) ^{\frac{1}{q}}.
\end{eqnarray*}%
We know that for $\alpha \in \left[ 0,1\right] $ and $\forall t_{1},t_{2}\in %
\left[ 0,1\right] $, 
\begin{equation*}
\left\vert t_{1}^{\alpha }-t_{2}^{\alpha }\right\vert \leq \left\vert
t_{1}-t_{2}\right\vert ^{\alpha },
\end{equation*}%
therefore%
\begin{eqnarray*}
\dint\limits_{0}^{1}\left\vert t^{\alpha }-\left( 1-t\right) ^{\alpha
}\right\vert ^{p}dt &\leq &\dint\limits_{0}^{1}\left\vert 1-2t\right\vert
^{\alpha p}dt \\
&=&\dint\limits_{0}^{\frac{1}{2}}\left[ 1-2t\right] ^{\alpha
p}dt+\dint\limits_{\frac{1}{2}}^{1}\left[ 2t-1\right] ^{\alpha p}dt \\
&=&\frac{1}{\alpha p+1}.
\end{eqnarray*}%
Since $\left\vert f^{\prime }\right\vert ^{q}$ is prequasiinvex on $\left[
a,a+\eta (b,a)\right] ,$ we have inequality (\ref{2-8}), which completes the
proof.
\end{proof}

\begin{remark}
a) If in Theorem \ref{2.4}, we let $\eta (b,a)=b-a$ then inequality (\ref%
{2-8}) become inequality (\ref{1-4}) of Theorem\ref{1.d}.

b) If in Theorem \ref{2.4}, we let $\alpha =1$then inequality (\ref{2-8})
become inequality (\ref{1-7}) of Theorem\ref{1.b}.

c) In Theorem \ref{2.4}, assume that $\eta $ satisfies condition C. Using
inequality (\ref{2-11}) we get%
\begin{eqnarray*}
&&\left\vert \frac{f(a)+f\left( a+\eta (b,a)\right) }{2}-\frac{\Gamma
(\alpha +1)}{2\eta ^{\alpha }(b,a)}\left[ J_{a^{+}}^{\alpha }f(a+\eta
(b,a))+J_{\left( a+\eta (b,a)\right) ^{-}}^{\alpha }f(a)\right] \right\vert
\\
&\leq &\frac{\eta (b,a)}{2\left( \alpha p+1\right) ^{\frac{1}{p}}}\left(
\max \left\{ \left\vert f^{\prime }(a)\right\vert ^{q},\left\vert f^{\prime
}(a+\eta (b,a))\right\vert ^{q}\right\} \right) ^{\frac{1}{q}}.
\end{eqnarray*}
\end{remark}

\end{document}